\documentclass[a4paper,11pt,twoside]{amsart}
\usepackage[utf8]{inputenc}
\usepackage{amssymb,amsmath,amsthm,amscd}
\usepackage{enumerate}
\usepackage{verbatim}
\usepackage{lmodern}
\usepackage{color}
\usepackage[all,cmtip]{xy}

\usepackage{hyperref}

\newcommand{\C}{\ensuremath{\mathbb{C}}}
\newcommand{\R}{\ensuremath{\mathbb{R}}}
\renewcommand\epsilon\varepsilon

\theoremstyle{definition}
\newtheorem{thm}{Theorem}[section]
\newtheorem{dfn}[thm]{Definition}
\newtheorem{lem}[thm]{Lemma}
\newtheorem{prp}[thm]{Proposition}

\newtheorem{rem}[thm]{Remark}

\author{Tim de Laat}
\address{Tim de Laat
\newline Westf\"alische Wilhelms-Universit\"at M\"unster, Mathematisches Institut
\newline Einsteinstra\ss{}e 62, 48149 M\"unster, Germany}
\email{tim.delaat@uni-muenster.de}

\author{Timo Siebenand}
\address{Timo Siebenand
\newline Westf\"alische Wilhelms-Universit\"at M\"unster, Mathematisches Institut
\newline Einsteinstra\ss{}e 62, 48149 M\"unster, Germany}
\email{timo.siebenand@uni-muenster.de}

\title[Exotic group $C^{*}$-algebras of simple Lie groups]{Exotic group $C^{*}$-algebras of simple Lie groups with real rank one}

\begin{document}

\begin{abstract}
Exotic group $C^*$-algebras are $C^*$-algebras that lie between the universal and the reduced group $C^*$-algebra of a locally compact group. We consider simple Lie groups $G$ with real rank one and investigate their exotic group $C^{*}$-algebras $C^*_{L^{p+}}(G)$, which are defined through $L^p$-integrability properties of matrix coefficients of unitary representations. First, we show that the subset of equivalence classes of irreducible unitary $L^{p+}$-representations forms a closed ideal of the unitary dual of these groups. This result holds more generally for groups with the Kunze-Stein property. Second, for every classical simple Lie group $G$ with real rank one and every $2 \leq q < p \leq \infty$, we determine whether the canonical quotient map $C^*_{L^{p+}}(G) \twoheadrightarrow C^*_{L^{q+}}(G)$ has non-trivial kernel. Our results generalize, with different methods, recent results of Samei and Wiersma on exotic group $C^*$-algebras of $\mathrm{SO}_{0}(n,1)$ and $\mathrm{SU}(n,1)$. In particular, our approach also works for groups with property (T).
\end{abstract}

\maketitle

\section{Introduction and main results} \label{sec:introduction}
With every locally compact group $G$, we can associate two natural group $C^{*}$-algebras, namely the universal group $C^{*}$-algebra $C^{*}(G)$ and the reduced group $C^{*}$-algebra $C^{*}_{r}(G)$. The left-regular representation $\lambda \colon G \to \mathcal{B}(L^2(G))$ of $G$ extends to a quotient map $C^{*}(G) \twoheadrightarrow C^{*}_r(G)$ and this quotient map is a ${}^*$-isomorphism if and only if $G$ is amenable.

If $G$ is not amenable, there may be many group $C^{\ast}$-algebras lying\linebreak ``between'' the universal and the reduced group $C^{*}$-algebra. An exotic group $C^{*}$-algebra of a locally compact group $G$ is a $C^{*}$-completion $A$ of $C_c(G)$ such that the identity map from $C_c(G)$ to itself extends to non-injective quotient maps from $C^{*}(G)$ to $A$ and from $A$ to $C^{*}_r(G)$:
\[
	C^{*}(G) \twoheadrightarrow A \twoheadrightarrow C^{*}_r(G).
\]

In recent years, exotic group $C^{*}$-algebras and related constructions, such as exotic crossed products, have received an increased amount of attention, partly because of their relation with the Baum-Connes conjecture (see e.g.~\cite{MR3514939}, \cite{MR3824785}, \cite{MR3837592}).

The systematic study of exotic group $C^{*}$-algebras (of discrete groups) goes back to Brown and Guentner \cite{MR3138486}, who introduced and studied the notion of ideal completion. If $\Gamma$ is a countable discrete group and $D$ is an appropriate algebraic two-sided ideal of $\ell^{\infty}(\Gamma)$ (e.g.~$D=\ell^p(\Gamma)$, with $1 \leq p \leq \infty$), the ideal completion $C^{*}_D(\Gamma)$, which is defined as the completion of the group ring of $\Gamma$ with respect to the natural norm defined through all unitary representations with sufficiently many matrix coefficients lying in the ideal $D$, is a potentially exotic group $C^{*}$-algebra. Okayasu proved that for $2 \leq q < p \leq \infty$, the canonical quotient map $C^{*}_{\ell^p}(\mathbb{F}_d) \twoheadrightarrow C^{*}_{\ell^q}(\mathbb{F}_d)$ between ideal completions of the non-abelian free group $\mathbb{F}_d$ has non-trivial kernel \cite{MR3238088}. This result was independently obtained by Higson and by Ozawa, and it was extended to discrete groups containing a non-abelian free group as a subgroup by Wiersma \cite{MR3705441}. It is an open question whether every non-amenable (discrete) group admits exotic group $C^*$-algebras.

In the setting of (non-discrete) locally compact groups, the $L^p$-integrability properties (for different values of $p$) of matrix coefficients of unitary representations of the group also form an important source of potentially exotic group $C^*$-algebras. For $p \in [1, \infty]$, a unitary representation $\pi \colon G \to \mathcal{B}(\mathcal{H})$ is called an $L^{p}$-representation if suitable many of its matrix coefficients are elements of $L^p(G)$ (see Definition \ref{dfn:lprep} for the precise definition). The representation $\pi$ is called an $L^{p+}$-representation if $\pi$ is an $L^{p+\varepsilon}$-representation for every $\varepsilon > 0$.

The aim of this article is to investigate the exotic group $C^*$-algebras $C^*_{L^{p+}}(G)$ (see Section \ref{sec:clpplus} for the precise construction), which are constructed from the $L^{p+}$-representations of $G$, for connected simple Lie groups with real rank one and finite center. To this end, we first prove a structural result on the subset of the unitary dual $\widehat{G}$ consisting of (irreducible unitary) $L^{p+}$-representations of such groups, which more generally holds for groups with the Kunze-Stein property (see Section \ref{sec:ksgroups}). This result captures a key idea of Samei and \mbox{Wiersma} (cf.~\cite[Theorem 5.3]{sameiwiersma2}) in the setting of the Fell topology.

First, recall that a subset $S\subset \widehat{G}$ is called an ideal if for every representation $\pi \in S$ and every unitary representation $\rho$ of $G$, the unitary representation $\pi \otimes \rho$ is weakly contained in $S$.
\begin{thm} \label{thm:KSFell}
	Let $G$ be a Kunze-Stein group, and let $\widehat{G}_{L^{p+}}$ denote the subset of $\widehat{G}$ consisting of (equivalence classes of) $L^{p+}$-representations. Then $\widehat{G}_{L^{p+}}$ is a closed ideal of $\widehat{G}$.
\end{thm}
This result was already known for the groups $\mathrm{SO}_0(n,1)$ and $\mathrm{SU}(n,1)$, with $n \geq 2$, from the work of Shalom \cite[Theorem 2.1]{MR1792293}. Related results were shown in \cite{MR0560837}, \cite{MR777342}, \cite{MR1355801}, \cite{MR1391214}, \cite{MR1682805}, \cite{MR1905394}.

Theorem \ref{thm:KSFell} leads to a natural strategy to find and distinguish exotic group $C^*$-algebras. The idea is to determine for which values of $p$, the ideals $\widehat{G}_{L^{p+}}$ are pairwise different. To this end, it suffices to show that there are representations with sufficiently many matrix coefficients which are $L^p$-integrable for certain $p$, but not $L^q$-integrable when $p > q$.

Our second result is an application of this approach in the setting of simple Lie groups with real rank one. Let $G$ be a connected simple Lie group with real rank one and finite center. Then $G$ is locally isomorphic to one of the following groups: $\mathrm{SO}_{0}(n,1)$, $\mathrm{SU}(n,1)$, $\mathrm{Sp}(n,1)$ (with $n \geq 2$) or $F_{4(-20)}$. The first three are called the (connected) classical simple Lie groups with real rank one, whereas $F_{4(-20)}$ is an exceptional Lie group.

Given a locally compact group $G$, we define $\Phi(G)$ as follows:
\begin{align*}
	\Phi(G) := \inf \{ p\in [1,\infty] \mid \forall \,\pi \in \widehat{G}\setminus \{ \tau_0 \},\, \pi \mbox{ is an } L^{p+}\mbox{-representation}\},
\end{align*}
where $\tau_0$ denotes the trivial representation of $G$. The constant $\Phi(G)$ is known for the three aforementioned classical Lie groups and is given by
\begin{align} \label{eq:intparameters}
	\Phi(G) = \begin{cases}
	\infty \qquad &\textrm{if }\; G = \mathrm{SO}_0(n,1),\\
	\infty \qquad &\textrm{if }\; G = \mathrm{SU}(n,1),\\
	2n+1 \qquad &\textrm{if }\; G = \mathrm{Sp}(n,1).
	\end{cases}
\end{align}
The cases $\mathrm{SO}_0(n,1)$ and $\mathrm{SU}(n,1)$ essentially follow from Harish-Chandra's rich work. It also follows directly from Proposition \ref{prp:int_of_spher_func}. The number $\Phi(\mathrm{Sp}(n,1))$ was computed in \cite{MR1355801}.

Our second result characterizes the exotic group $C^*$-algebras of the type $C^*_{L^{p+}}(G)$ for the classical simple Lie groups with real rank one.
\begin{thm} \label{thm:classicalrealrankone}
	Let $G$ be a (connected) classical simple Lie group with real rank one. Then for $2 \leq q < p \leq \Phi(G)$ (where $\Phi(G)$ is as in \eqref{eq:intparameters}), the canonical quotient map
	\[
		C^*_{L^{p+}}(G) \twoheadrightarrow C^*_{L^{q+}}(G)
	\]
	has non-trivial kernel.	Furthermore, for every $p,q\in [\Phi(G),\infty)$, we have 
	\begin{align*}
		C^*_{L^{p+}}(G) =  C^*_{L^{q+}}(G).
	\end{align*}
	Furthermore, if $H$ is a connected simple Lie group with finite center that is locally isomorphic to a classical simple Lie group $G$ with real rank one, then the same result holds for $H$ with $\Phi(H)=\Phi(G)$.
\end{thm}
Additional to Theorem \ref{thm:classicalrealrankone}, we obtain partial results for the exceptional Lie group $F_{4(-20)}$ (see Theorem \ref{thm:exoticF4}).

Exotic group $C^*$-algebras of Lie groups were considered before in \cite{MR3418075}, in which Wiersma proved that for $2 \leq q < p \leq \infty$, the quotient map $C^*_{L^{p+}}(\mathrm{SL}(2,\mathbb{R})) \twoheadrightarrow C^*_{L^{q+}}(\mathrm{SL}(2,\mathbb{R}))$ has non-trivial
kernel, by studying the representation theory of $\mathrm{SL}(2,\mathbb{R})$. Note that $\mathrm{SL}(2,\mathbb{R})$ is locally isomorphic to $\mathrm{SO}_0(2,1)$, and hence included in Theorem \ref{thm:classicalrealrankone}. More recently, Samei and Wiersma deduced the existence of continua of exotic group $C^*$-algebras for certain groups having the ``integrable Haagerup property'' and the rapid decay property or the Kunze-Stein property. They obtained the cases $G=\mathrm{SO}_0(n,1)$ and $\mathrm{SU}(n,1)$ of Theorem \ref{thm:classicalrealrankone} above. Their method, which is of geometric nature, is inherently unable to deal with groups with Kazhdan's property (T) (see \cite{MR2415834} for background on property (T)), examples of which are $\mathrm{Sp}(n,1)$ and $F_{4(-20)}$. In the methods used in this article, the (integrable) Haagerup property does not play a role, and our method works equally well for the groups $\mathrm{Sp}(n,1)$ and $F_{4(-20)}$.

The article is organized as follows. After recalling some preliminaries in Section \ref{sec:preliminaries}, we recall the algebras $C^*_{L^{p+}}(G)$ and prove some new results about them in Section \ref{sec:clpplus}. Section \ref{sec:ksgroups} is concerned with the unitary dual and the algebras $C^*_{L^{p+}}(G)$ of Kunze-Stein groups $G$. In particular, we prove Theorem \ref{thm:KSFell} in that section. In Section \ref{sec:realrankone}, we prove Theorem \ref{thm:classicalrealrankone}.

\section*{Acknowledgements}
We thank Siegfried Echterhoff, Ebrahim Samei and Matthew Wiersma for interesting discussions and useful comments.

The authors are supported by the Deutsche Forschungsgemeinschaft - Project-ID 427320536 - SFB 1442, as well as under Germany’s Excellence Strategy - EXC 2044 - 390685587, Mathematics Münster: Dynamics - Geometry - Structure.

\section{Preliminaries} \label{sec:preliminaries}

\subsection{Matrix coefficients and weak containment}
Let $G$ be a locally compact group, and let $\pi \colon G \to \mathcal{B}(\mathcal{H})$ be a unitary representation. Recall that a matrix coefficient of $\pi$ is a function from $G$ to $\mathbb{C}$ of the form $\pi_{\xi,\eta} \colon s \mapsto \langle \pi(s)\xi,\eta \rangle$, where $\xi,\eta \in \mathcal{H}$. Such functions are continuous bounded functions on $G$. Matrix coefficients of the form $\pi_{\xi,\xi}$ (i.e.~with $\xi=\eta$) are called diagonal matrix coefficients.

A unitary representation $\pi_1$ of $G$ is said to be weakly contained in a unitary representation $\pi_2$ of $G$ if every diagonal matrix coefficient of $\pi_1$ can be approximated by finite sums of diagonal matrix coefficients of $\pi_2$ uniformly on compact subsets of $G$. For details, we refer to \cite{MR0458185}.

\subsection{Unitary dual and Fell topology} \label{subsec:unitarydualfelltopology}
Let $G$ be a locally compact group, and let $\widehat{G}$ denote its unitary dual, i.e.~the set of equivalence classes of irreducible unitary representations equipped with the Fell topology. If $S$ is a subset of $\widehat{G}$, then the closure $\overline{S}$ of $S$ in the Fell topology consists of all elements of $\widehat{G}$ which are weakly contained in $S$. Let $\widehat{G}_r$ denote the subspace of $\widehat{G}$ consisting of all elements of $\widehat{G}$ which are weakly contained in the left regular representation $\lambda \colon G \to \mathcal{B}(L^2(G))$. For details, we refer to \cite{MR0458185}.

\subsection{Constructions of exotic group $C^{*}$-algebras} \label{subsec:constructions}
Let us recall two well-known constructions of exotic group $C^{*}$-algebras, which go back to \cite{MR3514939} and \cite{MR3141810}.

Let $G$ be a locally compact group, and let $\mu_G$ be a Haar measure on $G$. A group $C^*$-algebra associated with $G$  is a $C^{*}$-completion $A$ of $C_c(G)$ with respect to a $C^*$-norm $\|.\|_{\mu}$ which satisfies $\|f\|_u \geq \|f\|_{\mu} \geq \|f\|_r$ for all $f \in C_c(G)$, where $\|f\|_u$ and $\|f\|_r$ denote the universal and the reduced $C^*$-norm, respectively. In this case, the identity map $C_c(G) \to C_c(G)$ induces canonical surjective ${}^*$-homomorphisms $C^{*}(G) \twoheadrightarrow A$ and $A \twoheadrightarrow C^{*}_r(G)$. If both these quotient maps have non-trivial kernel, then $A$ is called an exotic group $C^*$-algebra. This is equivalent to the definition in Section \ref{sec:introduction}.

Let $G$ be a locally compact group, and let $\widehat{G}$ and $\widehat{G}_r$ be as before. A subset $S \subset \widehat{G}$ is called admissible if $\widehat{G}_r \subset \overline{S}$. If $S$ is admissible, then
\[
	\|f\|_S:=\sup\{\|\pi(f)\| \mid \pi \in S\}
\]
defines a $C^{*}$-norm on $C_c(G)$. The corresponding completion $C^*_S(G)$ is a potentially exotic group $C^{*}$-algebra. Furthermore, if $S$ is admissible, we have $\widehat{C^*_S(G)} =\overline{S}$, where $\widehat{C^*_S(G)}$ is the spectrum of the group $C^*$-algebra $C^*_S(G)$. More precisely, there is a bijective map from $S$ to $\widehat{C^*_S(G)}$ mapping a representation $\pi\colon G\to \mathcal{B}(\mathcal{H})$ to the corresponding irreducible ${}^*$-representation 
$\pi_{*}\colon C^*_S(G)\to \mathcal{B}(\mathcal{H})$ given by
$\pi_{*}(f) = \int f(s)\pi(s)\, \mathrm{d}\mu_G(s)$ for $f\in C_c(G)$.

A subset $S\subset \widehat{G}$ is said to be an ideal if for every representation $\pi \in S$ and every unitary representation $\rho$ of $G$, the unitary representation $\pi \otimes \rho$ is weakly contained in $S$. Note that the Fell absorption principle implies that every non-empty ideal $S \subset \widehat{G}$ is admissible. For more details on the above construction, we refer to \cite{MR3514939}.

Let us now discuss another approach, from \cite{MR3141810}. Recall that the Fourier-Stieltjes algebra $B(G)$, consisting of all matrix coefficients of unitary representations of $G$, is a subalgebra of the algebra of continuous bounded functions on $G$. It can be canonically
identified with the dual space $C^*(G)^*$ of the universal group $C^*$-algebra $C^*(G)$ through the pairing induced by
\begin{align*}
	\langle \varphi, \, f \rangle = \int \varphi f \mathrm{d}\mu_G
\end{align*}
for $\varphi \in B(G)$ and $f\in C_c(G) \subset C^*(G)$. Furthermore, $B(G)$ admits, in a canonical way, a left and a right $G$-action (see \cite[Section 3]{MR3141810}). Let $B_r(G) \subset B(G)$ denote the dual space of the reduced group $C^*$-algebra $C^*_r(G)$. It was shown in \cite[Lemma 3.1]{MR3141810} that if $E \subset B(G)$ is a weak*-closed, $G$-invariant subspace of $B(G)$ containing $B_r(G)$, then
\begin{align*}
	C_E^*(G) = C^*(G)/{}^{\perp}E
\end{align*}
is a group $C^*$-algebra, where ${}^\perp E = \{ x\in C^*(G)\mid \langle\varphi, x\rangle = 0 \; \forall \varphi \in E\}$ is the pre-annihilator of $E$.

Let us explain the connection between the two approaches mentioned above. This connection is essentially contained in \cite{MR3141810}, but we give a proof which is in line with our framework and conventions.

\begin{prp}\label{prp:ideal_in_spec_coaction}
	A non-empty closed set $S \subset \widehat{G}$ is an ideal if and only if the canonical comultiplication
	$\Delta \colon C^*(G) \to \mathcal{M}(C^*(G) \otimes C^*(G))$ factors through a coaction
	$\Delta_S \colon C^*_S(G) \to \mathcal{M}(C^*_S(G)\otimes C^*(G) )$. 
	Here $C^*_S(G)\otimes C^*(G)$ denotes the minimal tensor product
	of the C*-algebra $C^*_S(G)$ and $C^*(G)$.
\end{prp}
\begin{proof}
First, suppose that $S$ is an ideal, and let $q\colon C^*(G)\to C^*_S(G)$ be the canonical quotient map and $\sigma \colon G \to \mathcal{B}(\mathcal{H})$ a unitary representation of $G$ such that the integrated form $\sigma_* \colon C^*(G)\to \mathcal{B}(\mathcal{H})$ is a faithful (non-degenerate) *-representation of $C^*(G)$. The *-representation $\pi_* = (\bigoplus_{\rho \in S} \rho)_* \otimes	\sigma_*$ is a faithful (non-degenerate) *-representation of $C^*_{S}(G)\otimes C^*(G)$. Therefore, it extends to a faithful *-representation of the multiplier algebra of the algebra $C^*_{S}(G)\otimes C^*(G)$, which we denote with $\pi_*$ again. Note that the \linebreak *-representation 
	${\pi_* \circ (q \otimes \mathrm{id}) \circ \Delta}$ of $C^*(G)$ is the integrated form of \linebreak $\bigoplus_{\rho \in S}\rho \otimes \sigma$. Since $S$ is assumed to be an ideal, this implies that $\pi_* \circ (q \otimes \mathrm{id}) \circ \Delta$ factors through the canonical quotient map $q\colon C^*(G)\to C^*_S(G)$. Finally, since $\pi_*$ is faithful,
	we obtain that $(q \otimes \mathrm{id}) \circ  \Delta$ factors through $q$, which proves the first direction.
	
	Now, suppose that $C^*_S(G)$ is a group $C^*$-algebra with $\widehat{C_S^*(G)} = S$
	such that the comultiplication
	$\Delta \colon C^*(G) \to \mathcal{M}(C^*(G) \otimes C^*(G))$ factors through a coaction
	$\Delta_S \colon C^*_S(G) \to \mathcal{M}(C^*_S(G)\otimes C^*(G) )$.  
	Let $\rho \in S$, and let $\pi$ be any unitary representation of $G$. Again, we denote the integrated form of $\rho$ and $\pi$ by $\rho_*$ and $\pi_*$, respectively. Then
	$\pi_*\otimes \rho_*$ is a non-degenerate\linebreak
	*-representation of $C^*(G)\otimes C^*_S(G)$. Hence, it extends to a non-degenerate\linebreak *-representation of the multiplier algebra $\mathcal{M}(C^*(G)\otimes
	C^*_S(G))$, which we denote with $\pi_*\otimes \rho_*$ again.  The *-representation
	$(\pi_*\otimes \rho_*) \circ \Delta_S$ is the integrated form of the unitary representation $\pi\otimes \rho$ of $G$.
	This shows that $\pi\otimes \rho$ is weakly contained in $S$. 
\end{proof}
Combining this proposition with \cite[Corollary 3.13]{MR3141810}, we obtain the following result, which is probably well known to experts, but to our knowledge not explicitly contained in the literature.
\begin{prp} \label{prp:idealfs}
	Let $G$ be a locally compact group and $C^*_\mu(G)$ a group $C^*$-algebra of $G$. The following are equivalent:
	\begin{enumerate}[(i)]
		\item The set $\widehat{C^*_\mu(G)}\subset \widehat{G}$ is a closed ideal in $\widehat{G}$.
		\item The dual space $C^*_\mu(G)^*$ of $C^*_\mu(G)$ is a $G$-invariant ideal in $B(G)$.
	\end{enumerate}
\end{prp}

\subsection{Covering groups of Lie groups} \label{subsec:coveringgroups}
A covering group of a connected Lie group $G$ is a Lie group $\widetilde{G}$ with a surjective Lie group homomorphism $\sigma \colon \widetilde{G} \rightarrow G$ in such a way that $(\widetilde{G},\sigma)$ is a (topological) covering space of $G$.

A universal covering space is a covering space which is simply connected. Every connected Lie group $G$ admits a universal covering space $\widetilde{G}$, which admits a canonical Lie group structure. Such a universal covering group $\widetilde{G}$ of $G$ is uniquely determined up to isomorphism. Universal covering groups satisfy the exact sequence $1 \rightarrow \pi_1(G) \rightarrow \widetilde{G} \rightarrow G \rightarrow 1$, where $\pi_1(G)$ denotes the fundamental group of $G$. For details, we refer to \cite[Section I.11]{MR1920389}.

\subsection{Class one representations and spherical functions}\label{subsec:class_one_spherical_func}
We briefly recall Gelfand pairs, spherical functions and class one representations, which we will need for the proof of Theorem \ref{thm:classicalrealrankone}.

Let $G$ be a locally compact group, let $\mu_G$ be a Haar measure, and let $K < G$ be a compact subgroup of $G$. A function $\varphi \colon G \to \mathbb{C}$ is called $K$-bi-invariant if $\varphi(k_1sk_2)=\varphi(s)$ for all $s \in G$ and $k_1,k_2 \in K$. Let $C_c(K \backslash G / K)$ denote the $*$-subalgebra of the convolution algebra $C_c(G)$ consisting of all $K$-bi-invariant functions on $G$ with compact support. The pair $(G,K)$ is called a Gelfand pair if the algebra $C_c(K \backslash G / K)$ is commutative.

Let $(G,K)$ be a Gelfand pair. A continuous $K$-bi-invariant function $\varphi \colon G \to \mathbb{C}$ is called a spherical function if
\[
	\chi \colon f \mapsto \int f(s)\varphi(s^{-1})d\mu_G(s)
\]
defines a non-trivial character of $C_c(K \backslash G / K)$. Here, $\mu_G$ denotes a Haar measure on $G$.

A pair $(G,K)$ consisting of a locally compact group $G$ with a compact subgroup $K$ is a Gelfand pair if and only if for every irreducible unitary representation $\pi \colon G \to \mathcal{B}(\mathcal{H})$, the subspace $\mathcal{H}^{K}$ of $\mathcal{H}$ consisting of $\pi(K)$-invariant vectors is at most one-dimensional. An irreducible unitary representation $\pi \colon G \to \mathcal{B}(\mathcal{H})$ for which $\dim(\mathcal{H}^K)=1$ is called a class one representation. For a Gelfand pair $(G,K)$, we write $(\widehat{G}_K)_1$ for the (equivalence classes of) class one representations in $\widehat{G}$. The space $(\widehat{G}_K)_1$ is also called the spherical unitary dual.

Let $\pi \colon G \to \mathcal{B}(\mathcal{H})$ be a class one representation of the Gelfand pair $(G,K)$, and let $\xi \in \mathcal{H}^K$ be a $\pi(K)$-invariant vector of norm one. Then the diagonal matrix coefficient $\pi_{\xi,\xi}$ is a positive definite spherical function. This assignment defines a bijection between $(\widehat{G}_K)_1$ and the set of all positive definite spherical functions of the Gelfand pair $(G,K)$.

For an introduction to the theory of Gelfand pairs, spherical functions and class one representations, we refer to \cite{MR2640609}.

\section{The algebras $C^*_{L^{p+}}(G)$} \label{sec:clpplus}
We now recall and discuss the algebras $C^*_{L^{p+}}(G)$, which are the (potentially) exotic group $C^*$-algebras of our interest. To this end, we first recall $L^p$-representations and $L^{p+}$-representations.

\begin{dfn} \label{dfn:lprep}
	Let $G$ be a locally compact group, let $\pi \colon G \to \mathcal{B}(\mathcal{H})$ be a unitary representation, and let $ p \in [1, \infty]$.
\begin{enumerate}[(i)]
	\item We say that $\pi$ is an $L^{p}$-representation if there is a dense subspace 
		$\mathcal{H}_0 \subset \mathcal{H}$
		such that $\pi_{\xi,\eta}\in L^p(G)$ for all $\xi,\eta \in \mathcal{H}_0$.
	\item We say that $\pi$ is an $L^{p+}$-representation if $\pi$ is an $L^{p+\varepsilon}$-representation for all $\varepsilon > 0$.
\end{enumerate}
\end{dfn}
Recall that whenever a function $f \in C_b(G)$ is contained in $L^p(G)$ for some $p \in [1,\infty]$, then $f$ is contained in $L^q(G)$ for all $q \geq p$. In particular, this is true for matrix coefficients of unitary representations.

\begin{rem}
Suppose that $\pi\colon G \to \mathcal{B}(\mathcal{H})$ is a unitary representation and that $p \in [1,\infty]$. It follows by the polarization identity that $\pi$ is an $L^p$-representation if and only if there is a dense subspace $\mathcal{H}_0\subset \mathcal{H}$ such that $\pi_{\xi,\xi}\in L^p(G)$ for all $\xi \in \mathcal{H}_0$. This definition of $L^p$-representation is used in \cite{sameiwiersma2}.

Also, note that in the literature different terminology, e.g.~strongly $L^p$- and strongly $L^{p+}$-representation, is used for what we call $L^p$-representation and $L^{p+}$-representation.
\end{rem}
The following result gives a sufficient condition for a cyclic representation to be an $L^p$-representation.
\begin{prp}\label{prp:cyclic_lp_reps}
Let $\pi \colon G \to \mathcal{B}(\mathcal{H})$ be a cyclic unitary representation with cyclic vector $\xi \in \mathcal{H}$ such that $\pi_{\xi,\xi}\in L^p(G)$ for some $p\in [1,\infty]$. Then $\pi$ is an $L^p$-representation.
\end{prp}
\begin{proof}
The case $p = \infty$ is trivial, so suppose that $p\in [1,\infty)$.

Since $\pi$ has cyclic vector $\xi \in \mathcal{H}$, the subspace $\mathcal{H}_0= \mathrm{span} \{ \pi(s)\xi \mid s\in G \}$ is dense in $\mathcal{H}$.
Let $s_1$ and $s_2$ be arbitrary elements of $G$, and let $\zeta_1 := \pi(s_1) \xi$ and $\zeta_2 := \pi(s_2) \xi$. Then
	\begin{align*}
		\int \vert \pi_{\zeta_1,\zeta_2} \vert ^p d\mu_G
		= \int \vert	\langle \pi(s_2^{-1} t s_1) \xi,\xi \rangle \vert^p d\mu_G(t)
		= \Delta_G (s_1^{-1}) \int \vert \pi_{\xi,\xi}\vert^p d\mu_G,
	\end{align*}
where $\mu_G$ denotes a Haar measure on $G$ and $\Delta_G$ denotes the associated modular function. Hence, $\pi_{\zeta_1,\zeta_2}\in L^p(G)$ for all $\zeta_1,\zeta_2 \in \lbrace \pi(s)\xi \mid
	s\in G\}$. This implies that $\pi_{\zeta_1,\zeta_2}\in L^p(G)$ for all $\zeta_1,\zeta_2 \in 
	\mathcal{H}_0$, which completes the proof.
\end{proof}

Recall the construction of the potentially exotic group $C^*$-algebra $C^*_S(G)$ from Section \ref{subsec:constructions}. Our main interest is to consider the subset $S$ of $\widehat{G}$ consisting of (equivalence classes of) $L^{p+}$-representations (or $L^p$-representations) of a locally compact group $G$. Note that in general, however, these sets may be empty, which is for example the case for non-compact locally compact abelian groups.

For a locally compact group $G$ and $p\in[2,\infty]$, let $C^*_{L^p}(G)$ and $C^*_{L^{p+}}(G)$ be the group $C^*$-algebras obtained as the completions of $C_c(G)$ with respect to the $C^*$-norms
\begin{align*}
	\|\cdot \|_{L^p} &\colon C_c(G)\to [0,\infty),\, f \mapsto 
	\sup \lbrace \|\pi(f)\| \mid \pi \mbox{ is a } L^p\mbox{-representation}\} \mbox { and}\\
	\|\cdot \|_{L^{p+}}& \colon C_c(G)\to [0,\infty),\, f \mapsto 
	\sup \{ \|\pi(f)\| \mid \pi \mbox{ is a } L^{p+}\mbox{-representation}\},
\end{align*}
respectively.

It is well known that $C^*(G) = C^*_{L^{\infty}}(G)$ and $C^*_r(G) = C^*_{L^2}(G)$. The following result follows directly from Proposition \ref{prp:idealfs}.
\begin{prp} \label{prp:dual_C_lp_ideal}
	Let $p\in [2,\infty]$. Then the dual spaces $\widehat{C^*_{L^{p}}(G)}$ and $\widehat{C^*_{L^{p+}}(G)}$ of $C^*_{L^{p}}(G)$ and $C^*_{L^{p+}}(G)$, respectively, are ideals in $\widehat{G}$.	
\end{prp}

\section{On the unitary dual of Kunze-Stein groups} \label{sec:ksgroups}
Recall that a locally compact group $G$ is called a Kunze-Stein group if the convolution product $m \colon C_c(G) \times C_c(G) \to C_c(G), \; (f,g) \mapsto f \ast g$ extends to a bounded bilinear map $L^q(G) \times L^2(G) \to L^2(G)$ for all $q \in [1,2)$.

This property originated from the work of Kunze and Stein \cite{MR0163988}, who showed the above property for the group $\mathrm{SL}(2,\mathbb{R})$. In 1978, Cowling proved that connected semisimple Lie groups with finite center are Kunze-Stein groups \cite{MR0507240}. Other classes of Kunze-Stein groups were provided in \cite{MR0507240}, \cite{MR0936361} and \cite{veca}.

The aim of this section is to prove Theorem \ref{thm:KSFell}, starting from the following result by Samei and Wiersma (see \cite[Theorem 5.3]{sameiwiersma2}).
	\begin{thm}[Samei-Wiersma]\label{thm:samei_wiersma_main}\label{thm:sameiwiersma}
Let $G$ be a Kunze-Stein group, and let $2 \leq p < \infty$. For a unitary representation $\pi \colon G \to \mathcal{B}(\mathcal{H})$, the following are equivalent:
\begin{enumerate}[(i)]
			\item The representation $\pi$ extends to a *-representation of $C^*_{L^{p+}}(G)$.
			\item We have $B_\pi \subset L^{p+\varepsilon}(G)$ for all $\varepsilon > 0$,
\end{enumerate}
where $B_\pi$ denotes the closure in the weak$^*$-topology (on $B(G)$) of the linear span of all matrix coefficients of $\pi$, i.e.~the functions $\pi_{\xi,\eta}$, with $\xi,\eta \in \mathcal{H}$.
\end{thm}
We first prove the following result.
\begin{prp}\label{prp:KS_int_irr_rep}
Let $G$ be a Kunze-Stein group, let $p \in [2,\infty]$, let $\pi$ be an irreducible unitary representation of $G$, and let $\xi \in \mathcal{H}$ be a nonzero vector. Then $\pi$ is an $L^{p+}$-representation if and only if $\pi_{\xi,\xi}\in L^{p+\varepsilon}(G)$ for all $\varepsilon > 0$.
\end{prp}
\begin{proof}
The case $p = \infty$ is trivial, so suppose $p\in [2,\infty)$. If $\pi$ is an $L^{p+}$-representation then $\pi_{\xi,\xi}\in L^{p+\varepsilon}(G)$ for all $\varepsilon > 0$ by Theorem \ref{thm:samei_wiersma_main}. On the other hand, suppose that $\pi_{\xi,\xi}\in L^{p+\varepsilon}(G)$ for all $\varepsilon > 0$. Since $\xi$ is a cyclic vector, the representation $\pi$ is an $L^{p+}$-representation by Proposition \ref{prp:cyclic_lp_reps}.
\end{proof}
\begin{prp} \label{prp:weakcon_and_int}
Let $G$ be a Kunze-Stein group, let $p\in [2,\infty)$, let $\pi$ be an $L^{p+}$-representation of $G$ and $\rho$ another unitary representation of $G$ which is weakly	contained in $\pi$. Then $\rho$ is an $L^{p+}$-representation.
\end{prp}
\begin{proof}
Let $\pi_*$ and $\rho_*$ denote the integrated forms of $\pi$ and $\rho$, respectively. Since $\pi$ is an $L^{p+}$-representation, $\pi_*$ factors through the canonical quotient map $C^*(G) \to C^*_{L^{p+}}(G)$. Since $\rho$ is weakly contained in $\pi$, we have that $\ker \pi_* \subset \ker \rho_*$. Hence $\rho_*$ factors through the quotient map $C^*(G) \to C^*_{L^{p+}}(G)$ as well. Equivalently,	$\rho$ extends to a *-representation of $C^*_{L^{p+}}(G)$. By Theorem \ref{thm:samei_wiersma_main}, it follows that $\rho$ is an $L^{p+}$-representation.
\end{proof}
For $p \in [2,\infty]$, let
\begin{align*}
		\widehat{G}_{L^p} &:= \{ [\pi]\in \widehat{G}\mid \pi \mbox{ is an } L^p\mbox{-representation} \}
		\mbox{ and }\\
		\widehat{G}_{L^{p+}} &:= \{ [\pi]\in \widehat{G}\mid \pi \mbox{ is an } L^{p+}\mbox{-representation} \}
\end{align*}
If $2 \leq q < p \leq \infty$, then we have
\begin{align*}
		\widehat{G}_{L^q} \subset \widehat{G}_{L^{q+}}\subset \widehat{G}_{L^{p}}.
\end{align*}
As mentioned before, the sets $\widehat{G}_{L^p}$ and $\widehat{G}_{L^{p+}}$ could, in general, be empty. However, for Kunze-Stein groups, we can prove Theorem \ref{thm:KSFell}. First, we identify the dual space $\widehat{C^*_{L_{p+}}(G)}$ with $\widehat{G}_{L^{p+}}$.
\begin{lem}\label{lem:lp_closed_prop}
		Let $G$ be a Kunze-Stein group and $p \in [2,\infty]$. Then
		\begin{align*}
			\widehat{C^*_{L_{p+}}(G)} = \widehat{G}_{L^{p+}}.
		\end{align*}
\end{lem}
\begin{proof}
		Let $\pi \in \widehat{G}_{L^{p+}}$ be an irreducible unitary representation. By definition of $C^*_{L^{p+}}(G)$, the representation $\pi$ extends to $C^*_{L^{p+}}(G)$. This implies
		that the integrated form $\pi_*$ of $\pi$ lies in
		$\widehat{C^*_{L^{p+}}(G)}$. 
		On the other hand, if $\pi_* \in \widehat{C^*_{L^{p+}}(G)}$
		then by Theorem \ref{thm:samei_wiersma_main}, $\pi$ is an $L^{p+}$-representation and hence $\pi \in \widehat{G}_{L^{p+}}$.
\end{proof}
\begin{proof}[Proof of Theorem \ref{thm:KSFell}]
The statement of the theorem now follows directly from Lemma \ref{lem:lp_closed_prop}, Proposition \ref{prp:dual_C_lp_ideal} and Proposition \ref{prp:weakcon_and_int} (see Section \ref{subsec:unitarydualfelltopology} for a description of the closure of a subset of $\widehat{G}$).
\end{proof}

\section{Simple Lie groups of real rank one} \label{sec:realrankone}
In this section, we investigate the group $C^*$-algebras $C^*_{L^{p+}}(G)$ for connected simple Lie groups $G$ with real rank one and finite center. In particular, we prove Theorem \ref{thm:classicalrealrankone}.

From now on, we assume $G$ to be a connected non-compact simple Lie group with finite center and $K$ a maximal compact subgroup of $G$. It is well known that a pair $(G,K)$ consisting of such groups is a Gelfand pair (see \cite[Corollary 1.5.6]{MR0954385} or \cite[Chapter VI, Theorem 1.1]{MR1834454}). Recall also that all maximal compact subgroups of $G$ are conjugate under an inner automorphism of $G$.

Let $\mathfrak{g}$ be the Lie algebra of $G$, let $\mathfrak{k}$ be the Lie algebra of $K$, and let $\mathfrak{g}= \mathfrak{k} \oplus \mathfrak{p}$
be the corresponding Cartan decomposition of $\mathfrak{g}$. After fixing a maximal abelian subalgebra $\mathfrak{a}$ of $\mathfrak{p}$ and choosing a positive root system $\Delta^+(\mathfrak{g},\mathfrak{a})$ from the root system $\Delta(\mathfrak{g},\mathfrak{a})$, we obtain the polar decomposition $G = K\overline{A^+}K$ of $G$, where $A^+ = \exp(\mathfrak{a}^+)$ and $\mathfrak{a}^+$ is the Weyl chamber corresponding to $\Delta^+(\mathfrak{g},\mathfrak{a})$.

The Haar measure $\mu_G$ on $G$ can be normalized in such a way that for every $f \in C_c(G)$, we have
\begin{align} \label{eq:polardecompOfInt}
	\int_G f d\mu_G = \int_K \int_{\overline{A^+}}\int_K f(k_1ak_2)J(a)\,d\mu_K(k_1)\,d\mu_A(a)\,d\mu_K(k_2)
\end{align}
where $\mu_K$ is the normalized Haar measure on $K$, the measure $\mu_A$ is
the Haar measure on $A$, and for $a \in A$,
\begin{align} \label{eq:j}
	J(a) = \prod_{\alpha \in \Delta^+(\mathfrak{g},\mathfrak{a})}\left(e^{\alpha(\log a)}- e^{-\alpha(\log a)} \right)^{\dim (\mathfrak{g}_\alpha)},
\end{align}
where $\mathfrak{g}_{\alpha}$ is the root space corresponding to $\alpha$ (see \cite[Proposition 2.4.6]{MR0954385}). For details on the assertions made above, we refer to \cite{MR0954385}, \cite{MR1834454} and \cite{MR1920389}.

There is an intimate relationship between the class one representations of a connected simple Lie group and the class one representations of a finite covering group of this Lie group, as is shown by the following result. This result is certainly known, but since we could not find an appropriate reference, we give a proof.
\begin{lem} \label{prp:bij_class_one_local_iso_grps}
Let $G$ be a connected non-compact simple Lie group with finite center, $K < G$ a maximal compact subgroup and $q \colon \widetilde{G} \to G$ a connected finite covering group. Then $\widetilde{G}$ is a connected non-compact simple Lie group, $\tilde{K} := q^{-1}(K)$ is a maximal compact subgroup of $\widetilde{G}$, and $q$ induces a bijection
\begin{align*}
		q^*\colon \left( \widehat{G}_K\right)_1 \to \left( \widehat{\widetilde{G}}_{\tilde{K}}\right)_1,\,
		[\pi]\mapsto [\pi\circ q].
\end{align*}
Furthermore, suppose that $p\in [2,\infty)$. Then $[\pi] \in \left( \widehat{G}_K\right)_1$ is represented by an $L^{p+}$-representation if and only if $q^*([\pi])$ is represented by an $L^{p+}$-representation.
\end{lem}
\begin{proof}
	Since $q$ is a finite covering, $q$ is a proper map, and hence $\tilde{K}=q^{-1}(K)$ is a compact subgroup of $\widetilde{G}$. Suppose $C$ is a compact subgroup of $\widetilde{G}$ with $\tilde{K}\subset C$. Then $q(C)$ is a compact subgroup of $G$ and
	$K=q(\tilde{K}) \subset q(C)$. The maximality of $K$ implies $K = q(C)$. Hence,
	$C \subset q^{-1}(K) = \tilde{K}$.
	
	Note that $q^*$ is well defined. Furthermore, $q^*$ is obviously injective. It remains to show that $q^*$ is surjective. To this end, let $\pi \colon \widetilde{G} \to \mathcal{B}(\mathcal{H}_{\pi})$ be a class one representation.	Then there is a $\tilde{K}$-invariant vector $\xi\in \mathcal{H}_\pi$ of norm one, and $\tilde{\omega} = \langle \pi(\cdot) \xi,\xi \rangle$ is a positive definite spherical function. Since $G$ is a quotient group of $\widetilde{G}$ by a subgroup $\Gamma < Z(\widetilde{G})$, we know that $\tilde{\omega}$ factors through $q\colon \widetilde{G}\to G$, by, say, $\omega\colon G \to \C$. It is immediate that $\omega$ is a positive definite normalized function as well. Let $\pi'\colon G \to \mathcal{B}(\mathcal{H})$ be the cyclic
	unitary representation with cyclic vector $\xi'\in \mathcal{H}$ of $G$ and
	$\omega = \langle \pi'(\cdot)\xi', \xi' \rangle$. Then $\pi' \circ q$ is a cyclic unitary representation
	with $\langle \pi' \circ q(\cdot) \xi',\xi' \rangle
	 = \omega \circ q = \tilde{\omega} = \langle \pi(\cdot) \xi,\xi \rangle$. Hence $\pi' \circ q$
	 and $\pi$ are unitary equivalent. This implies that $\pi$ factors through $q\colon \widetilde{G}\to G$,
	 which concludes the proof of the assertion that $q^*$ is a bijection.
	 
	 Let $\pi \colon G \to \mathcal{B}(\mathcal{H}_{\pi})$ be a class one representation, let $\xi \in \mathcal{H}_\pi$ be a $\pi(K)$-invariant vector of norm one, and let $\omega = \pi_{\xi,\xi}$ the associated positive definite spherical function. If $\pi$ is an $L^{p+}$-representation, then, due to Theorem \ref{thm:samei_wiersma_main}, $\omega \in L^{p+\varepsilon}(G)$ for all $\varepsilon > 0$. It follows that
	 $\omega\circ q \in L^{p+\varepsilon}(\widetilde{G})$ for all $\varepsilon > 0$, by the quotient integral formula for Haar integrals. Hence,
	 $\pi\circ q$ is an $L^{p+}$-representation (see Proposition \ref{prp:KS_int_irr_rep}). A similar argument shows the opposite direction.
\end{proof}
We now specialize to the real rank one case. From now on, let $G$ be a connected simple Lie group with real rank one and finite center. It is well known that such a $G$ is locally isomorphic to one of the following Lie groups: $\mathrm{SO}_{0}(n,1)$, $\mathrm{SU}(n,1)$, $\mathrm{Sp}(n,1)$, with $n \geq 2$, or to the exceptional Lie group $\mathrm{F}_{4(-20)}$. The three of these groups arise as the isometry groups of the classical rank one symmetric spaces of the non-compact type. Explicitly, they are given by
\begin{align*}
	\mathrm{SO}(n,1) &= \{g \in \mathrm{SL}(n+1,\mathbb{R}) \mid g^* I_{n,1} g = I_{n,1} \},\\
	\mathrm{SU}(n,1) &= \{g \in \mathrm{SL}(n+1,\mathbb{C}) \mid g^* I_{n,1} g = I_{n,1} \},\\
	\mathrm{Sp}(n,1) &= \{g \in \mathrm{GL}(n+1,\mathbb{H}) \mid g^* I_{n,1} g = I_{n,1} \},
\end{align*}
where $I_{n,1}=\mathrm{diag}(1,\ldots,1,-1)$, i.e.~the diagonal $(n+1) \times (n+1)$-matrix with the first $n$ diagonal entries equal to $1$ and the $n+1^{\textrm{th}}$ entry equal to $-1$. These three groups are called the classical simple Lie groups with real rank one. The group $\mathrm{F}_{4(-20)}$ is an exceptional Lie group. We refer to \cite{MR1920389} for more details.
\begin{rem}
The universal covering group of $\mathrm{SO}_{0}(n,1)$ ($n \geq 3$) has finite center, the group $\mathrm{Sp}(n,1)$ ($n \geq 2$) is itself simply connected (so it is its own universal covering group) and has finite center, and the group $\mathrm{F}_{4(-20)}$ is simply connected and has trivial center (see e.g.~\cite{MR1920389} or \cite{yokota}).

The universal covering group of $\mathrm{SU}(n,1)$ ($n \geq 2$) and $\mathrm{SO}(2,1)$ have center isomorphic to $\mathbb{Z}$. Since every non-trivial quotient group of $\mathbb{Z}$ is a finite group, every group which is locally isomorphic to $\mathrm{SU}(n,1)$ ($n \geq 2$) or $\mathrm{SO}(2,1)$ must either have finite center or be isomorphic to $\widetilde{\mathrm{SU}}(n,1)$ ($n \geq 2$) or to $\widetilde{\mathrm{SO}}(2,1)$. Because of the accidental local isomorphism $\mathrm{SO}(2,1) \approx \mathrm{SU}(1,1)$, it follows that if $G$ is a connected simple Lie group with real rank one and infinite center, then it is isomorphic to $\widetilde{\mathrm{SU}}(n,1)$ for some $n \geq 1$. For details, see e.g.~\cite{MR1920389}.
\end{rem}
Much of the theory recalled below goes back to Harish-Chandra's seminal work. For the purposes of this article, however, the exposition in the monograph by Gangolli and Varadarajan \cite{MR0954385} seems more suitable, and we use this monograph as a reference.

Let $G$ be a classical simple Lie group with real rank one, and let $\Phi(G)$ be as in \eqref{eq:intparameters}. Let $G=KAN$ be an Iwasawa decomposition of $G$, and let $P= MAN$ be a minimal parabolic subgroup. Here $M$ is the centralizer of $A$ in $K$. We write $\mathfrak{a}$ for the Lie algebra of $A$ and
$\mathfrak{a}^*$ (resp. $\mathfrak{a}^*_{\mathbb{C}}$) for the dual space $\hom_{\R}(\mathfrak{a},\R)$ (resp. $\hom_{\R}(\mathfrak{a},\C)$) of $\mathfrak{a}$
(resp. $\mathfrak{a}_{\mathbb{C}}$). Finally, let 
\begin{align*}
	\rho = \frac{1}{2} \sum_{\alpha\in\Delta^+(\mathfrak{g},\mathfrak{a})}\dim (\mathfrak{g}_\alpha)\, \alpha.
\end{align*}
The induced representation $\pi_\lambda $ of the character
\begin{align*}
	P= MAN \to \mathbb{C},\, man \mapsto e^{\lambda(\log a)}
\end{align*}
to $G$, with $\lambda \in \mathfrak{a}^*_{\mathbb{C}}$, yields a not necessarily unitary
group representation (which is unitary if $\lambda \in i\mathfrak{a}^*$) with a $K$-invariant
vector $\xi_\lambda$ of norm one. Hence, the matrix coefficient 
\begin{align*}
	\psi_\lambda = \langle \pi_\lambda(\cdot) \xi_\lambda, \xi_\lambda \rangle
\end{align*}
defines a spherical function for $(G,K)$.

In what follows, we identify $\mathfrak{a}^*_{\mathbb{C}}$ with $\mathbb{C}$ via the linear
map from $\mathfrak{a}_{\mathbb{C}}^*$ to $\mathbb{C}$ mapping an element $\gamma\in \Delta^+(\mathfrak{g},\mathfrak{a})$ with $\frac{1}{2}\gamma \not\in \Delta^+(\mathfrak{g},\mathfrak{a})$ to $1$.

As above let $\left(\widehat{G}_K\right)_1$ denote the set of all irreducible unitary class one representations. We recall the following result by Kostant \cite{MR0245725} (see also \cite[Lemma 5.2]{MR1779896}).
\begin{lem}\label{lem:kostant}
\mbox{}
\begin{enumerate}[(i)]
	\item The following holds:
		\begin{equation*}
		\rho =
		\begin{cases}
			\frac{n-1}{2} & \mbox{if } G = \mathrm{SO}_0(n,1),\\
			n& \mbox{if } G = \mathrm{SU}(n,1),\\
			2n+1 & \mbox{if } G = \mathrm{Sp}(n,1),\\
			11 & \mbox{if } G = \mathrm{F}_{4(-20)}.
		\end{cases}
	\end{equation*}
	\hspace{0pt}
	\item The mapping
	\begin{align*}
		i\mathfrak{a}^*_+\cup [0,\, \rho_0(G)) \to \left(\widehat{G}_K\right)_1 \setminus \{\tau_0\},\, \lambda \mapsto \pi_\lambda
	\end{align*}
	is bijective, where
	\begin{equation*}
		\rho_0(G) =
		\begin{cases}
			\frac{n-1}{2} & \mbox{if } G = \mathrm{SO}_0(n,1),\\
			n& \mbox{if } G = \mathrm{SU}(n,1),\\
			2n-1 & \mbox{if } G = \mathrm{Sp}(n,1),\\
			5 & \mbox{if } G = \mathrm{F}_{4(-20)}.
		\end{cases}
	\end{equation*}
\end{enumerate}
\end{lem}
We now characterize which of the class one representations $\pi_\lambda$ recalled above, with $\lambda \in [0,\rho_0(G))$, are $L^{p+}$-representations.
\begin{prp}\label{prp:int_of_spher_func}
	Let $G$ be $\mathrm{SO}_0(n,1)$, $\mathrm{SU}(n,1)$, $\mathrm{Sp}(n,1)$ or 
	$\mathrm{F}_{4(-20)}$, and
	let $\lambda \in [0,\rho_0(G))$. Then the class one representation $\pi_\lambda$ is an
	$L^{p+}$-representation if and only if $\lambda$ satisfies
	\[
		p(\rho - \lambda) \geq 2\rho.
	\]
\end{prp}
The proof of this result follows from \cite[Theorems 8.47 and 8.48]{MR0855239}. Specifically, the result we need can be found in \cite[p.~847--848]{MR1779896}. It follows from the known asymptotic behaviour of 
$\psi_\lambda = \langle \pi_\lambda(\cdot) \xi_\lambda,\xi_\lambda \rangle$ on $A^+$. For the 
convenience of the reader, we sketch the proof.
\begin{proof}
	The asymptotic behaviour of the spherical function $\psi_\lambda$ on $A^+$ is given by
	\begin{align*}
		\psi_{\lambda}(a) \sim e^{(\lambda-\rho)\log a}
	\end{align*}
	as $a\to \infty$ (see \cite[p. 847]{MR1779896}), and the asymptotic behaviour of the function $J$ from \eqref{eq:j} on $A^+$ is given by
	\begin{align*}
		J(a) \sim e^{2\rho \log a}
	\end{align*}
	as $a\to \infty$ (see \cite[p. 848]{MR1779896}). Using
	 (\ref{eq:polardecompOfInt}), it follows that
	\begin{align*}
		\psi_{\lambda} \in L^{p+\varepsilon}(G)
	\end{align*}
	for all $\varepsilon > 0$ and all $\lambda \in [0,\rho_0(G)]$ if and only if
	\[
		p(\rho - \lambda)\geq 2 \rho.
	\]
\end{proof}
For any Gelfand pair $(G,K)$, let $\Phi_K(G) \in [1,\infty]$ be defined as
\begin{align*}
	\Phi_K(G) := \inf\{p \in [1,\infty] \mid \forall \pi \in (\widehat{G}_K)_1\backslash \lbrace \tau_0 \rbrace,\, \pi \mbox{ is  an } L^{p+}\textrm{-representation}\},
\end{align*}
where $\tau_0$ denotes the trivial unitary representation of $G$.
It is clear that $\Phi_K(G) \leq \Phi(G)$.
\begin{prp}\label{prp:int_const_spherical_vs_global}
	Let $G$ be a connected simple Lie group with real rank one and finite center that is not locally isomorphic to $\mathrm{F}_{4(-20)}$. Then, we have $\Phi_K(G) = \Phi(G)$.
\end{prp}
\begin{proof}
	Suppose that $G$ is a connected Lie group with finite center locally isomorphic to $\mathrm{SO}_0(n,1)$ or $\mathrm{SU}(n,1)$. Then Lemma \ref{prp:bij_class_one_local_iso_grps} and Proposition
	\ref{prp:int_of_spher_func} immediately imply that $\Phi_K(G) = \infty\geq \Phi(G)$.
	
	From Proposition \ref{prp:int_of_spher_func} and the identity $\Phi(\mathrm{Sp}(n,1)) = 2n+1$, it follows that $\Phi(\mathrm{Sp}(n,1)) = \Phi_K(\mathrm{Sp}(n,1))$. Since $\mathrm{Sp}(n,1)$ is
	simply connected, every connected Lie group $G$ that is locally isomorphic to $\mathrm{Sp}(n,1)$
	is isomorphic to a quotient group of $\mathrm{Sp}(n,1)$, with $\mathrm{Sp}(n,1)$ as its universal covering group.
	Hence, we have $\Phi(G)\leq \Phi(\mathrm{Sp}(n,1)) = \Phi_K(\mathrm{Sp}(n,1)) = \Phi_K(G)$.
	Here the last equality follows by Lemma \ref{prp:bij_class_one_local_iso_grps} again.
\end{proof}
\begin{rem}
	Note that Proposition \ref{prp:int_const_spherical_vs_global} and Lemma \ref{prp:bij_class_one_local_iso_grps} imply that $\Phi(G) = \Phi(G')$ for locally isomorphic
	connected Lie groups $G$ and $G'$ with real rank one and finite center whenever $G$ is
	not locally isomorphic to $\mathrm{F}_{4(-20)}$.
\end{rem}
We now give the proof of Theorem \ref{thm:classicalrealrankone}.
\begin{proof}[Proof of Theorem \ref{thm:classicalrealrankone}]
	For the proof of the first assertion we only need to show that
	\begin{align*}
		\widehat{G}_{L^{p+}}\subsetneq \widehat{G}_{L^{q+}}
	\end{align*}
	for $2\leq p < q \leq \Phi(G)$ (see Lemma \ref{lem:lp_closed_prop}).
	To this end, let $p,q \in [2,\Phi(G)]$ with $p<q$.
	By Proposition \ref{prp:KS_int_irr_rep},
	an irreducible unitary representation is an $L^{p+}$-representation if and only if a non-trivial vector 
	state of the representation lies in $L^{p+\varepsilon}(G)$ for all $\varepsilon > 0$.
	Furthermore, the positive definite spherical functions are in one-to-one correspondence with class one
	irreducible unitary representations (see Section \ref{subsec:class_one_spherical_func}). 
	By Lemma \ref{lem:lp_closed_prop}, together with the fact that
	the GNS-construction of every positive definite spherical function is an irreducible unitary
	representation, it suffices to show that there is a 
	positive definite spherical function $\psi$ for the Gelfand pair $(G,K)$ that lies in $L^{p+\varepsilon}(G)$ for all $\varepsilon > 0$ but not in 
	$L^{q+\varepsilon}(G)$ for all $\varepsilon > 0$.
	Lemma \ref{prp:bij_class_one_local_iso_grps} implies that we can restrict ourselves
	to the classical cases, i.e.~the cases where $G$ is equal to $\mathrm{SO}_0 (n,1)$, $\mathrm{SU}(n,1)$ or
	$\mathrm{Sp}(n,1)$.	Now Lemma \ref{lem:kostant} and Proposition \ref{prp:int_of_spher_func} complete the proof of
	 the first assertion. The second assertion follows from the definition of $\Phi(G)$.
\end{proof}
Note that for every connected simple Lie group $G$ with finite center and property (T), it was already known from the work of Cowling \cite{MR0560837} that there exists a $p_0 \in [2,\infty)$ (depending on $G$) such that every matrix coefficient of a unitary representation of $G$ is in $L^p(G)$ for every $p \geq p_0$.

Let us now consider the exceptional Lie group $\mathrm{F}_{4(-20)}$. Unfortunately, we cannot obtain a description as complete as for the classical Lie groups as in Theorem \ref{thm:classicalrealrankone}. The reason for this is that we do not know what the value of $\Phi(\mathrm{F}_{4(-20)})$ is. However, from considering the class one representations of $\mathrm{F}_{4(-20)}$ in combination of a result of Cowling, we still obtain a partial result.
\begin{thm}	\label{thm:exoticF4}
	For $2 \leq q < p \leq \frac{11}{3}$, the canonical quotient map
	\[
		C^*_{L^{p+}}(\mathrm{F}_{4(-20)}) \twoheadrightarrow C^*_{L^{q+}}(\mathrm{F}_{4(-20)})
	\]
	has non-trivial kernel. Furthermore, for every $p,q\in [\frac{22}{3},\infty)$, we have 
	\[
		C^*_{L^{p+}}(\mathrm{F}_{4(-20)}) =  C^*_{L^{q+}}(\mathrm{F}_{4(-20)}).
	\]
\end{thm}
For the proof of this theorem, we will use the following Lemma, which is a special case of a result due to Cowling (see \cite[Lemme 2.2.6]{MR0560837}).
\begin{lem} \label{lem:cowling2p}
Let $G$ be a connected simple Lie group with finite center, and let $K$ be a maximal compact subgroup of $G$. Suppose that there exists a $p \in [2,\infty)$ such that all non-constant positive-definite spherical functions of the Gelfand pair $(G,K)$ belong to $L^{p+\varepsilon}(G)$ for all $\varepsilon > 0$. Then every non-trivial irreducible unitary representation of $G$ is an $L^{2p+}$-representation.
\end{lem}
\begin{proof}[Proof of Theorem \ref{thm:exoticF4}]
The first assertion follows in the same way as the first part of Theorem \ref{thm:classicalrealrankone}, relying on Lemma \ref{lem:kostant} and Proposition \ref{prp:int_of_spher_func}.

For the second assertion, we use Lemma \ref{lem:cowling2p}. Indeed, note that by Proposition \ref{prp:int_of_spher_func}, every non-trivial class one representation of $G=F_{4(-20)}$ is an $L^{\frac{11}{3}+}$-representation, so every non-constant positive-definite spherical function belongs to $L^{\frac{11}{3}+\varepsilon}$ for all $\varepsilon > 0$. Hence, by the lemma, every non-trivial irreducible unitary representation of $G$ is an $L^{\frac{22}{3}+}$-representation, which implies the second assertion.
\end{proof}
In order to improve Theorem \ref{thm:exoticF4}, one would need to study the asymptotic behaviour of the isolated series representations.


\begin{thebibliography}{BEW18}

\bibitem[BGW16]{MR3514939}
P.~Baum, E.~Guentner, R.~Willett, \emph{Expanders, exact crossed products, and the Baum-Connes conjecture},
Ann.~$K$-Theory \textbf{1} (2016), 155--208. 

\bibitem[BHV08]{MR2415834}
B.~Bekka, P.~de~la Harpe and A~Valette, \emph{Kazhdan's Property (T)},
Cambridge University Press, Cambridge, 2008.

\bibitem[BG13]{MR3138486}
N.P.~Brown and E.P.~Guentner, \emph{New {$C^*$}-completions of discrete groups and related spaces},
Bull.~Lond.~Math.~Soc.~\textbf{45} (2013), 1181--1193.

\bibitem[BEW17]{MR3837592}
A.~Buss, S.~Echterhoff and R,~Willett, \emph{Exotic crossed products}, Operator algebras and applications -- the {A}bel {S}ymposium 2015, 67--114, Abel Symp., 12, Springer, 2017.

\bibitem[BEW18]{MR3824785}
\bysame, \emph{Exotic crossed products and the {B}aum-{C}onnes conjecture},
J.~Reine Angew.~Math.~\textbf{740} (2018), 111--159.

\bibitem[Cow78]{MR0507240}
M.~Cowling, \emph{The Kunze-Stein phenomenon},
Ann.~Math.~(2) \textbf{107} (1978),\linebreak 209--234.

\bibitem[Cow79]{MR0560837}
\bysame, \emph{Sur les coefficients des repr\'{e}sentations unitaires des groupes de Lie simples}, Analyse harmonique sur les groupes de Lie (S\'em., Nancy-Strasbourg 1976–1978), II, 132--178, Lecture Notes in Math., 739, Springer, Berlin, 1979.

\bibitem[vD09]{MR2640609}
G.~van Dijk, \emph{Introduction to harmonic analysis and generalized Gelfand pairs}, Walter de Gruyter \& Co., Berlin, 2009.

\bibitem[Dix77]{MR0458185}
J.~Dixmier, \emph{$C^*$-Algebras}, North-Holland Publishing Co., Amsterdam -- New York -- Oxford, 1977. 

\bibitem[GV88]{MR0954385}
R.~Gangolli, V.S.~Varadarajan, \emph{Harmonic Analysis of Spherical Functions on Real Reductive Groups}, Springer-Verlag, Heidelberg, 1988.

\bibitem[Hel01]{MR1834454}
S.~Helgason, \emph{Differential Geometry, Lie groups, and Symmetric Spaces}, Amer. Math. Soc. Providence, RI, 2001.

\bibitem[How82]{MR777342}
R.~Howe, \emph{On a notion of rank for unitary representations of the classical groups}, Harmonic analysis and group representations, 223--331, Liguori,
Naples, 1982.

\bibitem[KLQ13]{MR3141810}
S.~Kaliszewski, M.~Landstad and J.~Quigg, \emph{Exotic group $C^*$-algebras in noncommutative duality},
New York J.~Math.~\textbf{19} (2013), 689--711.

\bibitem[Kna86]{MR0855239}
A.W.~Knapp, \emph{Representation Theory of Semisimple Groups. An Overview Based on Examples}, Princeton University Press, Princeton, NJ, 1986. 

\bibitem[Kna02]{MR1920389}
\bysame, \emph{{Lie Groups Beyond an Introduction}}, {Birkh\"auser}, Boston, 2002.

\bibitem[Kos69]{MR0245725}
B.~Kostant, \emph{On the existence and irreducibility of certain series of representations},
Bull. Amer. Math. Soc. \textbf{75} (1969), 627--642.

\bibitem[KS60]{MR0163988}
R.A.~Kunze and E.M.~Stein, \emph{Uniformly bounded representations and harmonic analysis of the $2 \times 2$ real unimodular group},
Amer.~J.~Math.~\textbf{82} (1960), 1--62.

\bibitem[Li95]{MR1355801}
J.-S.~Li, \emph{The minimal decay of matrix coefficients for classical
  groups}, Harmonic analysis in {C}hina, Math. Appl., 327, 146--169, Kluwer Acad.
  Publ., Dordrecht, 1995.
  
\bibitem[LZ96]{MR1391214}
J.-S.~Li and C.-B.~Zhu, \emph{On the decay of matrix coefficients for exceptional groups},
Math.~Ann.~\textbf{305} (1996), 249--270.

\bibitem[Neb88]{MR0936361}
C.~Nebbia, \emph{Groups of isometries of a tree and the Kunze-Stein phenomenon}, Pacific J.~Math.~\textbf{133} (1988), 141--149. 

\bibitem[Oh98]{MR1682805}
H.~Oh, \emph{Tempered subgroups and representations with minimal decay of matrix coefficients},
Bull.~Soc.~Math.~France \textbf{126} (1998), 355--380. 

\bibitem[Oh02]{MR1905394}
\bysame, \emph{Uniform pointwise bounds for matrix coefficients of unitary representations and applications to {K}azhdan constants}, Duke Math.~J.~\textbf{113} (2002),\linebreak 133--192.

\bibitem[Oka14]{MR3238088}
R.~Okayasu, \emph{Free group {$C^*$}-algebras associated with {$\ell_p$}},
Internat.~J.~Math.~\textbf{25} (2014).

\bibitem[SW18]{sameiwiersma2}
E.~Samei and M.~Wiersma, \emph{Exotic {$C^{*}$}-algebras of geometric groups},
preprint, arXiv:1809.07007.

\bibitem[Sha00a]{MR1779896}
Y.~Shalom, \emph{Explicit {K}azhdan constants for representations of   semisimple and arithmetic groups}, Ann.~Inst.~Fourier (Grenoble) \textbf{50} (2000), 833--863.

\bibitem[Sha00b]{MR1792293}
\bysame, \emph{Rigidity, unitary representations of semisimple groups, and fundamental groups of manifolds with rank one transformation group},
Ann.~Math.~(2) \textbf{152} (2000), 113--182.

\bibitem[Vec02]{veca}
A.~Veca, \emph{The Kunze-Stein Phenomenon}, PhD Thesis, 2002, see \texttt{http://web.maths.unsw.edu.au/$\sim$michaelc/Students/veca.pdf}.

\bibitem[Wie15]{MR3418075}
M.~Wiersma, \emph{$L^p$-Fourier and Fourier-Stieltjes algebras for locally compact groups}, J.~Funct.~Anal.~\textbf{269} (2015), 3928--3951.

\bibitem[Wie16]{MR3705441}
\bysame, \emph{Constructions of exotic group {$C^*$}-algebras},
Illinois J.~Math.~\textbf{60} (2016), 655--667.

\bibitem[Yok09]{yokota}
I.~Yokota, \emph{Exceptional Lie groups}, preprint (2009), arXiv:0902.0431.

\end{thebibliography}
\end{document}